\documentclass{amsart}

\usepackage{amssymb}
\usepackage{amsthm}
\usepackage{amsmath}
\usepackage{amsbsy}
\usepackage[all]{xy}
\usepackage{bm}
\usepackage{hyperref}
\usepackage{tikz}
\usepackage{array}
\usepackage{colortbl,xcolor}
\usepackage{ytableau}
\usepackage{hhline}
\allowdisplaybreaks

\newtheorem{theorem}{Theorem}[section]
\newtheorem{proposition}[theorem]{Proposition}

\newtheorem{conjecture}[theorem]{Conjecture}
\newtheorem{lemma}[theorem]{Lemma}

\DeclareMathOperator{\Res}{Res}

\begin{document}

\title[]{On the number of simultaneous core \\partitions with $d$-distinct parts} 

\author[]{Noah Kravitz}
\address[]{Grace Hopper College, Yale University, New Haven, CT 06510, USA}
\email{noah.kravitz@yale.edu}

\begin{abstract} 
We investigate the number $N_{d,r}(s)$ of $(s, s+r)$-core integer partitions with $d$-distinct parts.  Our first main result is a proof of a recurrence relation conjectured by Sahin in 2018.  We also derive generating functions, asymptotics, and exact formulas for $N_{d,r}(s)$ when $r$ is within $d$ of a multiple of $s$.  Finally, we exhibit a surprising connection to $A$-restricted compositions.
\end{abstract}
\maketitle

\section{Introduction and main results}

A \textit{partition} of a nonnegative integer $n$ is a finite nonincreasing sequence of positive integers $\lambda=(\lambda_1, \lambda_2, \dots, \lambda_k)$ such that $n=\lambda_1+\lambda_2+\dots+\lambda_k$.  (The unique partition of $0$ is the empty partition.)  We say that $n$ is the \textit{size} of $\lambda$ and $\lambda_1, \lambda_2, \dots, \lambda_k$ are its \textit{parts}.  The study of integer partitions dates back at least to Euler and has since then become a staple of modern combinatorics and number theory.
\\

Partitions are often represented visually as \textit{Young diagrams}.  The Young diagram of $\lambda=(\lambda_1, \lambda_2, \dots, \lambda_k)$ consists of $k$ rows of left-justified cells where there are $\lambda_i$ cells in the $i$-th row.  The \textit{hook} of the cell in the $i$-th row (counting from the top) and the $j$-th column (counting from the left) consists of that cell and all of the cells directly below it or to the right of it in the Young diagram; the corresponding \textit{hook length} (written $h(i,j)$) is the total number of cells in this hook.  Figure \ref{hooks} shows the Young diagram for $\lambda=(8, 6, 3, 1)$ with the hook lengths written in the corresponding cells.  For a positive integer $s$, we say that $\lambda$ is \textit{$s$-core} if it has no hook of length $s$.  By extension, we say that $\lambda$ is $(s_1, s_2, \dots, s_m)$-core if it is $s_i$-core for each $s_i$.  (See \cite{AHJ} for a motivation of this definition.)

\begin{figure}[h]
\begin{ytableau}
11 &9 &8 &6 &5 &4 &2 &1\\
8 &6 &5 &3 &2 &1\\
4 &2 &1\\
1
\end{ytableau}
\caption{The partition $\lambda=(8, 6, 3, 1)$ is $s$-core for all positive integers $s \notin \{1, 2, 3, 4, 5, 6, 8, 9, 11\}$.}
\label{hooks}
\end{figure}

Simultaneous core partitions have garnered substantial interest ever since Anderson's seminal proof \cite{anderson} in 2002 that there are finitely many $(s, t)$-core partitions exactly when $s$ and $t$ are relatively prime, in which case the number of these partitions is the so-called rational Catalan number $\frac{1}{s+t}\binom{s+t}{s}$.  Olsson and Stanton \cite{olsson} showed that the largest such partition is unique and has size $\frac{(s^2-1)(t^2-1)}{24}$.  Other results in this area are due to Amdeberhan and Leven \cite{amdeberhan}, Yang, Zhong, and Zhou \cite{YZZ}, Aggarwal \cite{aggarwal}, and Wang \cite{wang}.
\\

A growing corpus of recent work on simultaneous core partitions with distinct parts can be traced back to the conjecture of Amdeberhan \cite{fibonacci} that the number of $(s, s+1)$-core partitions with distinct parts is the Fibonacci number $F_{s+1}$.  This conjecture was proven by Straub \cite{straub} and Xiong \cite{xiong}, and other results in a similar spirit can be found in Nath and Sellers \cite{nath}, Zaleski \cite{zaleski}, and Yan, Qin, Jin, and Zhou \cite{yan}.
\\

The property of having distinct parts can be generalized: for a positive integer $d$, we say that $\lambda=(\lambda_1, \lambda_2, \dots, \lambda_k)$ \textit{has $d$-distinct parts} if $\lambda_i-\lambda_{i+1}\geq d$ for all $1 \leq i \leq k-1$.  This definition, due originally to Alder \cite{alder}, has inspired work by Andrews \cite{andrews} and Alfes, Jameson, and Oliver \cite{AJO}, among others.
\\

Sahin \cite{sahin} combined these concepts in his analysis of simultaneous core partitions with $d$-distinct parts.  For positive integers $d$, $r$, and $s$, let $N_{d,r}(s)$ denote the number of $(s, s+r)$-core partitions with $d$-distinct parts.  Sahin derives the following recurrence relation for the case $r=1$.

\begin{theorem}[Sahin]
For any positive integer $d$, we have
$$N_{d,1}(s)=
\begin{cases}
s, &1\leq s \leq d+1\\
N_{d,1}(s-1)+N_{d,1}(s-d-1), &s \geq d+2.
\end{cases}$$
\label{sahinthm}
\end{theorem}

Furthermore, he conjectures that a similar relation holds whenever $r\leq d$.

\begin{conjecture}[Sahin]
For any positive integers $r\leq d$, we have
$$N_{d, r}(s)=
\begin{cases}
s, &1\leq s \leq d\\
s+r-1, &s=d+1\\
N_{d,r}(s-1)+N_{d,r}(s-d-1), &s \geq d+2.
\end{cases}$$
\label{sahinconj}
\end{conjecture}

Conjecture \ref{sahinconj} serves as a jumping-off point for our investigation of simultaneous core partitions with $d$-distinct parts.  In Section $2$, we present the main tools of this paper: the $\beta$-set associated with a partition $\lambda$; a natural extension of $N_{d,r}(s)$ to $s \leq 0$; and the $r$-Reduction Theorem.

\newtheorem*{reduction}{Theorem \ref{reduction}}
\begin{reduction}[$r$-Reduction]
For any positive integers $r \leq d$, we have
$$N_{d, r}(s)=
\begin{cases}
N_{d,1}(s), &1 \leq s \leq d\\
N_{d,1}(s)+(r-1)N_{d, 1}(s-2d), &s\geq d+1.
\end{cases}$$
\end{reduction}

In Section $3$, we derive an exact formula for all $N_{d,r}(s)$ with $r \leq d$ and discuss a connection to integer compositions with restricted part sizes.

\newtheorem*{generalformula}{Theorem \ref{generalformula}}
\begin{generalformula}
For any positive integers $r \leq d$ and any positive integer $s$, we have
$$N_{d,r}(s)=\sum_{\mu=0}^{\left\lceil \frac{s-1}{d+1} \right\rceil}\binom{s+d-d\mu-1}{\mu}+(r-1)\sum_{\mu=0}^{\left\lceil \frac{s-2d-1}{d+1} \right\rceil}\binom{s-d-d\mu-1}{\mu}.$$
\end{generalformula}

In Section $4$, we show how Conjecture \ref{sahinconj} can be deduced from the $r$-Reduction Theorem.  In Section $5$, we find the ordinary generating functions for $\{N_{d,r}(s)\}_{s=1}^{\infty}$ (for $r \leq d$) and the corresponding asymptotics.  We defer the bulk of the computations for the asymptotics, however, to Appendix A.

\newtheorem*{genfunc}{Theorem \ref{genfunc}}
\begin{genfunc}
For any positive integers $r \leq d$, the generating function $G_{d,r}(x)=\sum_{s=1}^{\infty}N_{d,r}(s)x^s$ is given by
$$G_{d,r}(x)=\frac{x(1+(r-1)x^d-rx^{d+1})}{(1-x-x^{d+1})(1-x)}.$$
\end{genfunc}

\newtheorem*{asymptotic}{Theorem \ref{asymptotic}}
\begin{asymptotic}
For each fixed pair of positive integers $r \leq d$, we have the asymptotic
$$N_{d,r}(s) \sim_s \frac{w_d^3(1+(r-1)w_d^{2d})}{(1-w_d)(d+1-dw_d)} \left(\frac{1}{w_d}\right)^s,$$
where $w_d\approx 1-\frac{\log(d+1)}{d+1}$ is the unique positive real root of $x^{d+1}+x-1$.
\end{asymptotic}

In Section $6$, we consider $N_{d,r}(s)$ for $r>d$.  We discuss results similar to those of Sections $3$ through $5$ for the case where $r$ is within $d$ of a multiple of $s$, and we explain why the general problem of $r>d$ is fundamentally more difficult than the case of $r \leq d$.

\section{Useful tools and preliminary results}

The first part of this section summarizes existing results on $\beta$-sets of partitions and relates these techniques to the problem at hand.  The second part describes the simplifications that are achieved by extending $N_{d,r}(s)$ to $s \leq 0$.  The third part proves the $r$-Reduction Theorem, which allows us to focus on the case $r=1$.

\subsection{Properties of $\beta$-sets}

For a partition $\lambda=(\lambda_1, \lambda_2, \dots, \lambda_k)$, the associated \textit{$\beta$-set} is defined to be $\beta(\lambda)=\{h(1,1), h(2, 1), \dots, h(k,1)\}$.  In other words, $\beta(\lambda)$ is the set of hook lengths that appear in the first column of the Young diagram of $\lambda$.  (The reader may recognize $\beta(\lambda)$ as the set of beads in the abacus diagram associated with $\lambda$.)  For example, we can see from Figure \ref{hooks} that $\beta(8, 6, 3, 1)=\{11, 8, 4, 1\}$.  An early instance of this now-ubiquitous technique appeared in Anderson \cite{anderson}.
\\

It is easy to see that $h(i,1)=\lambda_i+k-i$ and $\lambda_i=h(i,j)+i-k$, where $\lambda$ has $k$ parts.  Hence, the map from the set of all partitions to the set of finite subsets of the positive integers defined by $\lambda \mapsto \beta(\lambda)$ is a bijection.  (This map takes the empty partition of $0$ to the empty set.)  Because subsets of the positive integers are generally easier to work with than partitions, it is advantageous to express the $(s, s+r)$-core and $d$-distinct parts conditions in terms of $\beta$-sets.
\\

To this end, we present the following well-known ``abacus-condition'' lemma, which appears in \cite{anderson} and \cite{olsson}, among other places.

\begin{lemma}
For any positive integer $s$ and any partition $\lambda$, the following conditions are equivalent:
\begin{itemize}
\item The partition $\lambda$ is $s$-core.
\item For all $x \in \beta(\lambda)$ with $x \geq s$, we also have $x-s \in \beta(\lambda)$.
\end{itemize}
\label{abacus}
\end{lemma}

Since the elements of $\beta$ are strictly positive, this latter condition implies $s \notin \beta(\lambda)$.
\\

For a positive integer $d$, we say that a subset $S \subseteq \mathbb{Z}$ is \textit{$d$-th order twin-free} if $|x-y|>d$ for all distinct elements $x, y \in S$.  The following straightforward result appears in \cite{sahin}.

\begin{lemma}[Sahin]
For any positive integer $d$ and any partition $\lambda$, the following conditions are equivalent:
\begin{itemize}
\item The partition $\lambda$ has $d$-distinct parts.
\item The set $\beta(\lambda)$ is $d$-th order twin-free.
\end{itemize}
\label{twinfree}
\end{lemma}

We can use these two results to re-state our problem completely in terms of $\beta$-sets.  

\begin{lemma}

For positive integers $d$, $r$, and $s$, let $X_{d,r,s}$ denote the family of finite subsets $\beta$ of the positive integers satisfying the following three conditions:
\begin{enumerate}
\item For all $x \in \beta$ with $x \geq s$, we also have $x-s \in \beta$.
\item For all $x \in \beta$ with $x \geq s+r$, we also have $x-(s+r) \in \beta$.
\item The set $\beta$ is $d$-th order twin-free.
\end{enumerate}
Then $X_{d,r,s}$ consists of exactly the $\beta$-sets corresponding to the $(s, s+r)$-core partitions with $d$-distinct parts.  Moreover, $N_{d,r}(s)=|X_{d,r,s}|$ if $|X_{d,r,s}|$ is finite, and otherwise $N_{d,r}(s)=|X_{d,r,s}|=\infty$.
\label{correspondence}
\end{lemma}

\begin{proof}
Fix any $\beta \in X_{d,r,s}$, and let $\lambda$ be its associated partition.  By Lemma \ref{abacus}, the first and second conditions on $\beta$ are equivalent to $\lambda$ being $(s, s+r)$-core.  By Lemma \ref{twinfree}, the third condition is equivalent to $\lambda$ having $d$-distinct parts.  Now, fix any $(s, s+r)$-core partition $\lambda$ with $d$-distinct parts.  By the same reasoning, its $\beta$-set is an element of $X_{d,r,s}$.  This establishes the desired correspondence, and the second part of the lemma immediately follows.
\end{proof}

Recall Anderson's result \cite{anderson} that $N_{0,r}(s)<\infty$ if and only if $\gcd(s,r)=1$.  We prove an analogous criterion for $d\geq 1$.

\begin{lemma}
For any positive integers $d$, $r$, and $s$, we have $N_{d,r}(s)<\infty$ if and only if $\gcd(s,r)\leq d$.
\label{finite}
\end{lemma}

\begin{proof}
Lemma \ref{correspondence} tells us that $N_{d,r}(s)<\infty$ exactly when $X_{d,r,s}$ contains a finite number of elements.
First, suppose $\gcd(s,s+r)=\gcd(s,r)\leq d$.  By B\'ezout's Lemma, there exist positive integers $a$ and $b$ such that $as+b(s+r)\equiv \gcd(s,r) \pmod{s}$.  In particular, there exists a positive integer $c$ such that $b(s+r)=\gcd(s,r)+cs$.  We now claim that any $\beta \in X_{d,r,s}$ satisfies $\beta \subseteq \{1, 2, \dots, b(s+r)\}$.  Assume (for the sake of contradiction) that $x \in \beta$ for some $x\geq b(s+r)+1$.  Then, by Lemma \ref{correspondence}, we also have $x-b(s+r), x-cs \in \beta$.  But $|(x-cs)-(x-b(s+r))|=\gcd(s,r)\leq d$ contradicts $\beta$ being $d$-th order twin-free.  So we must have $\beta \subseteq \{1, 2, \dots, b(s+r)\}$, which implies that $|X_{d,r,s}|\leq 2^{b(s+r)}<\infty$.
\\

Second, suppose $\gcd(s,r)>d$.  It is clear that for any positive integer $a$, the set $\beta=\{1, 1+\gcd(s,r), 1+2\gcd(s,r), \dots, 1+a\gcd(s,r)\}$ satisfies all three conditions for being in $ X_{d,r,s}$.  Hence, $|X_{d,r,s}|=\infty$.
\end{proof}

This lemma tells us that $N_{d,r}(s)$ is always finite when $r \leq d$.  In fact, in this case we can get a tight bound on the largest possible element of any $\beta \in X_{d,r,s}$.  (The bound is tight in the sense that it is always achievable.)

\begin{lemma}
For any positive integers $r \leq d$ and any positive integer $s$, any set $\beta \in X_{d,r,s}$ must satisfy $\beta \subseteq \{1, 2, \dots, s+r-1\}\setminus \{s\}$.  When $r=1$, this bound gives $\beta \subseteq \{1, 2, \dots, s-1\}$.
\label{betabounds}
\end{lemma}

\begin{proof}
Assume (for contradiction) that $x \in \beta$ for some $x \geq s+r+1$.  Then, by Lemma \ref{correspondence}, we also have $x-s, x-(s+r) \in X_{d,r,s}$.  But $|(x-s)-(x-(s-r))|=r\leq d$ is a contradiction.  Hence, $\beta \subseteq \{1, 2, \dots, s+r\}$.  From the remark after Lemma \ref{abacus}, we also know that $s, s+r \notin \beta$, which yields $\beta \subseteq \{1, 2, \dots, s+r-1\}\setminus \{s\}$.  For the case of $r=1$, we note that $\{1, 2, \dots, s+1-1\}\setminus \{s\}=\{1, 2, \dots, s-1\}$.
\end{proof}

When $r=1$, the bound of Lemma \ref{betabounds} guarantees that the first two conditions of Lemma \ref{correspondence} are satisfied, so $N_{d,1}(s)$ simply counts the $d$-th order twin-free subsets of $\{1, 2, \dots, s-1\}$.

\subsection{Interpreting $N_{d,r}(s)$ for $s \leq 0$}

It is not a priori obvious how $N_{d,r}(s)$ should be extended to $s \leq 0$.  After all, hook lengths are always strictly positive, so avoiding hooks of length $s$ does not seem meaningful when $s \leq 0$.  Motivated by the previous section, we propose the following interpretation: for $r \leq d$, we define $N_{d,r}(s)$ to be the number of subsets of $\{1, 2, \dots, s+r-1\}$ satisfying the three conditions of Lemma \ref{correspondence}.  The following proposition expresses this quantity in a convenient form.

\begin{proposition}
For positive integers $r \leq d$, the extension of $N_{d,r}(s)$ to $s \leq 0$ has the form
$$N_{d,r}(s)=
\begin{cases}
1, &s<0\\
N_{d,1}(r), &s=0. 
\end{cases}$$
\label{negatives}
\end{proposition}

\begin{proof}
It is clear that when $s<0$, only the empty set satisfies the conditions of Lemma \ref{correspondence}: if there were some $x \in \beta$, then $\{x-s, x-2s, \dots\} \subseteq \beta$ would contradict the upper bound on the elements of $\beta$.  When $s=0$, the first and second conditions of Lemma \ref{correspondence} are always satisfied because $N_{d,r}(0)$ counts subsets of $\{1, 2, \dots, r-1\}$.  Thus, $N_{d,r}(0)$ counts the $d$-th order twin-free subsets of $\{1, 2, \dots, r-1\}$, of which there are $N_{d,1}(r)$ (as noted in the discussion of Lemma \ref{betabounds}).
\end{proof}

We can now re-state Theorem \ref{sahinthm} and Conjecture \ref{sahinconj} to include $s \leq 0$.

\begin{theorem}[Sahin, extended]
For any positive integer $d$, we have
$$N_{d,1}(s)=
\begin{cases}
1, &s=1\\
N_{d,1}(s-1)+N_{d,1}(s-d-1), &s \geq 2.
\end{cases}$$
\label{sahinthmext}
\end{theorem}

\begin{proof}[Proof of equivalence]
We need to show that $N_{d,1}(s-1)+N_{d,1}(s-d-1)=s$ for $2 \leq s \leq d+1$.  Note that $s-d-1\leq (d+1)-d-1=0$ and hence $N_{d,1}(s-d-1)=1$ by Proposition \ref{negatives}.  Induction on $s$ establishes the desired equality.
\end{proof}

\begin{conjecture}[Sahin, extended]
For any positive integers $r\leq d$, we have
$$N_{d, r}(s)=
\begin{cases}
1, &s=1\\
N_{d,r}(s-1)+N_{d, r}(s-d-1), &s\geq 2.
\end{cases}$$
\label{sahinconjext}
\end{conjecture}

\begin{proof}[Proof of equivalence]
The equivalence for $2 \leq s \leq d$ follows as in the previous proof.  For $s=d+1$, we must show that $N_{d,r}(s-1)+N_{d,r}(s-d-1)=s+r-1$, i.e., $N_{d,r}(d)+N_{d,r}(0)=d+r$.  But this follows immediately from $N_{d,r}(d)=d$ and $N_{d,r}(0)=N_{d,1}(r)=r$.
\end{proof}

These reformulations are substantially simpler than the originals, especially for the conjecture.  These results, along with what follows, should convince the reader that Proposition \ref{negatives} gives the ``correct'' extension of $N_{d,r}(s)$.

\subsection{The $r$-Reduction Theorem}

In this section, we prove the $r$-Reduction Theorem.  As the name suggests, this theorem helps us understand the dependence of $N_{d,r}(s)$ on $r$ (for $r \leq d$).  Indeed, we have a simple expression for $N_{d,r}(s)$ in terms of $N_{d,1}(s)$ and $N_{d,1}(s-2d)$.  Because the behavior of $N_{d,r}(s)$ depends only minimally on $r$, results for $N_{d,1}(s)$ often generalize to all $r \leq d$ with minimal effort.

\begin{theorem}[$r$-Reduction]
For any positive integers $r \leq d$, we have
$$N_{d, r}(s)=
\begin{cases}
N_{d,1}(s), &1 \leq s \leq d\\
N_{d,1}(s)+(r-1)N_{d, 1}(s-2d), &s\geq d+1.
\end{cases}$$
\label{reduction}
\end{theorem}

\begin{proof}
Since the result is trivial for $r=1$, we restrict our attention to $r \geq 2$.  We count the sets $\beta \in X_{d,r,s}$, which, by Lemma \ref{correspondence}, will give us $N_{d,r}(s)$.
\\

First, consider $1 \leq s \leq d$.  It is clear that $\beta$ cannot contain any element $x\geq s+1$, for then we would have $x-s \in \beta$, and $|(x)-(x-s)|=s\leq d$ would yield a contradiction.  Hence, $\beta \subseteq \{1, 2, \dots, s-1\}$.  Such a set $\beta$ trivially satisfies the first and second conditions of Lemma \ref{correspondence}, so $\beta$ can be any $d$-th order twin-free subset of $\{1, 2, \dots, s-1\}$.  As noted after Lemma \ref{betabounds}, there are exactly $N_{d,1}(s)$ such subsets.
\\

Second, consider $s \geq d+1$.  Lemma \ref{betabounds} tells us that $\beta \subseteq \{1, 2, \dots, s-1, s+1, \dots, s+r-1\}$.  We condition on the largest element of $\beta$.  If $\beta \subseteq \{1, 2, \dots, s-1\}$, then there are $N_{d,1}(s)$ possibilities.  Now, suppose that $\beta$ contains some element larger than $s-1$, say, $s+k \in \beta$ for some $1 \leq k \leq r-1$.  Then we also have $(s+k)-(s)=k \in \beta$.  Because $\beta$ is $d$-th order twin-free, we know that $k$ and $s+k$ are the only elements of $\beta$ in $\{k-d, k-d+1, \dots, k+d\} \cup \{s+k-d, \dots, s+k+d\}$.  Moreover, $k \leq r-1\leq d-1$ implies $k-d \leq -1$, and $k \geq 1$ implies $s+k+d \geq s+1+d \geq s+r+1$.  This lets us conclude that all other elements of $\beta$ must be in $\{k+d+1, k+d+2, \dots, s+k-d-1\}$.  (In particular, $s+k-d-1 \leq s+(d-1)-d-1=s-2$ shows that $\beta$ cannot contain a second element larger than $s-1$.)  Since $s+k-d-1<s$, the first and second conditions of Lemma \ref{correspondence} do not put any restrictions on which elements of $\{k+d+1, k+d+2, \dots, s+k-d-1\}$ can be in $\beta$.  In fact, $\beta$ can contain any $d$-th order twin-free subset of $\{k+d+1, k+d+2, \dots, s-k-d-1\}$.  Writing $\{k+d+1, k+d+2, \dots, s+k-d-1\}=\{(k+d)+1, (k+d)+2, \dots, (k+d)+(s-2d-1)\}$, we see that there are $N_{d,1}(s-2d)$ such subsets for each choice of $k$.  (For $s-2d\leq 0$, recall that $N_{d,1}(s-2d)=1$ counts only the empty set.)  In total, this gives $N_{d,r}(s)=N_{d,1}(s)+(r-1)N_{d,1}(s-2d)$, as desired.
\end{proof}

\pagebreak

\section{An exact formula for $N_{d,r}(s)$ and a connection \\to $A$-restricted compositions}

In this section, we use a direct counting argument to derive a formula for $N_{d,r}(s)$ when $r \leq d$.  We begin with the case $r=1$, and the corresponding formula for general $r \leq d$ follows from the $r$-Reduction Theorem.
\\

For any $x \in \mathbb{R}$, let $\lceil x \rceil$ denote the smallest integer greater than or equal to $x$.

\begin{lemma}
For any positive integer $d$ and any integer $s\geq -d+1$, we have
$$N_{d,1}(s)=\sum_{\mu=0}^{\left\lceil \frac{s-1}{d+1} \right\rceil}\binom{s+d-d\mu-1}{\mu}.$$
\label{r=1formula}
\end{lemma}

\begin{proof}
As in the proof of the $r$-Reduction Theorem, we count the sets $\beta \in X_{d,1,s}$.  We know from Lemma \ref{betabounds} that $\beta$ can be any $d$-th order twin-free subset of $\{1, 2, \dots, s-1\}$.  Suppose $\beta$ contains exactly $\mu$ elements.  Since the tightest packing occurs when consecutive elements of $\beta$ differ by exactly $d+1$, we see that $\mu$ ranges from $0$ to $\left\lceil \frac{s-1}{d+1} \right\rceil$.  The twin-free condition means that each $x \in \beta$ comes with a ``tail'' of elements $\{x+1, x+2, \dots, x+d\}$ that cannot be in $\beta$.  If we consider each $x$ and its tail to be a single block of $d+1$ elements, then $\beta$-sets with $\mu$ elements correspond to ways of filling $\{1, 2, \dots, s+d-1\}$ with $\mu$ blocks of length $d+1$ (representing the elements of $\beta$) and $(s+d-1)-\mu(d+1)$ single spaces (representing gaps between the blocks).  We have $(\mu)+((s+d-1)-\mu(d+1))=s+d-d\mu-1$ total objects, so there are $\binom{s+d-d\mu-1}{\mu}$ ways to choose the locations of the blocks.  Summing over all possible values of $\mu$ gives $$N_{d,1}(s)=\sum_{\mu=0}^{\left\lceil \frac{s-1}{d+1} \right\rceil}\binom{s+d-d\mu-1}{\mu}.$$
\end{proof}

For small values of $s$, we get simple formulas:

$$N_{d,1}(s)=
\begin{cases}
s, &2 \leq s \leq d+2\\
s+\binom{s-d-1}{2}, &d+3 \leq s \leq 2d+3.
\end{cases}$$

When $d=1$, Lemma \ref{r=1formula} gives
$$N_{1,1}(s)=\sum_{\mu=0}^{\left\lceil \frac{s-1}{2} \right\rceil}\binom{s+1-\mu-1}{\mu}=\sum_{\mu=0}^{\left\lceil \frac{s-1}{2} \right\rceil}\binom{s-\mu}{\mu}.$$
Using standard combinatorial arguments (see, e.g., \cite{count}, pg. 4), we can recognize the right-most expression as the Fibonacci number $F_{s+1}$, in agreement with other recent results \cite{straub}, \cite{xiong}, \cite{sahin}.
\\

This proof can be thought of as exhibiting a bijection between $X_{d,1,s}$ and the set of compositions of $s+d-1$ into parts of sizes $1$ and $d+1$.  Formally, given a subset $A\subseteq \mathbb{Z}^+$, an \textit{$A$-restricted composition} of a nonnegative integer $n$ is a finite sequence of elements of $A$ that sum to $n$.  These compositions have been studied in a variety of settings (see, e.g., \cite{a-restricted}, \cite{sills}, \cite{banderier}), and Chinn and Heubach \cite{1k} have paid special attention to the case $A=\{1,k\}$.  All of our results for $N_{d,1}(s)$ apply equally well to the number of $\{1, d+1\}$-restricted compositions of $s+d-1$.
\\

As promised, the $r$-Reduction Theorem makes the transition to general $r \leq d$ easy.

\begin{theorem}
For any positive integers $r \leq d$ and any positive integer $s$, we have
$$N_{d,r}(s)=\sum_{\mu=0}^{\left\lceil \frac{s-1}{d+1} \right\rceil}\binom{s+d-d\mu-1}{\mu}+(r-1)\sum_{\mu=0}^{\left\lceil \frac{s-2d-1}{d+1} \right\rceil}\binom{s-d-d\mu-1}{\mu}.$$
\label{generalformula}
\end{theorem}

\begin{proof}
Note that $\sum_{\mu=0}^{\left\lceil \frac{s-1}{d+1} \right\rceil}\binom{s+d-d\mu-1}{\mu}=N_{d,1}(s)$.  By the $r$-Reduction Theorem, all that remains to show is
$$\sum_{\mu=0}^{\left\lceil \frac{s-2d-1}{d+1} \right\rceil}\binom{s-d-d\mu-1}{\mu}=
\begin{cases}
0, &1 \leq s \leq d\\
N_{d,1}(s-2d), &s\geq d+1.
\end{cases}$$

For $1 \leq s \leq d$, we can compute $\left\lceil \frac{s-2d-1}{d+1} \right\rceil \leq \left\lceil \frac{d-2d-1}{d+1} \right\rceil=-1$, which means that $\sum_{\mu=0}^{\left\lceil \frac{s-2d-1}{d+1} \right\rceil}\binom{s-d-d\mu-1}{\mu}=0$, as desired.
\\

For $s \geq d+1$, we have $\sum_{\mu=0}^{\left\lceil \frac{s-2d-1}{d+1} \right\rceil}\binom{s-d-d\mu-1}{\mu}=\sum_{\mu=0}^{\left\lceil \frac{(s-2d)-1}{d+1} \right\rceil}\binom{(s-2d)+d-d\mu-1}{\mu}=N_{d,1}(s-2d)$ by Lemma \ref{r=1formula}.  This completes the proof.
\end{proof}

\section{Proof of Sahin's Conjecture}

In this section, we prove Conjecture \ref{sahinconjext} using the $r$-Reduction Theorem.

\begin{theorem}
For any positive integers $r\leq d$, we have
$$N_{d, r}(s)=
\begin{cases}
1, &s=1\\
N_{d,r}(s-1)+N_{d, r}(s-d-1), &s\geq 2.
\end{cases}$$
\label{recurrence}
\end{theorem}

\begin{proof}
Fix some $r\leq d$.  As usual, we count the sets $\beta \in X_{d,r,s}$.  The statement for $s=1$ is trivially true since Lemma \ref{betabounds} tells us that $\beta \subseteq \emptyset$.
\\

For $2 \leq s \leq d$, we get
\begin{align*}
N_{d,r}(s-1)+N_{d,r}(s-d-1)&=N_{d,1}(s-1)+N_{d,1}(s-d-1)\\
 &=N_{d,1}(s)\\
 &=N_{d,r}(s).
\end{align*}
The first equality uses the $r$-Reduction Theorem and the fact that $N_{d,r}(s)$ is uniformly $1$ for $s<0$.  The second equality follows from Theorem \ref{sahinthmext}, and the third comes from another application of the $r$-Reduction Theorem.  For $s=d+1$, we get
\begin{align*}
N_{d,r}((d+1)-1)+N_{d,r}((d+1)-d-1)&=N_{d,r}(d)+N_{d,r}(0)\\
 &=N_{d,1}(d)+N_{d,1}(r)\\
 &=d+r\\
 &=(d+1)+(r-1)\\
 &=N_{d,1}(d+1)+(r-1)N_{d,1}(-d+1)\\
 &=N_{d,r}(d+1).
\end{align*}
The third and fifth equalities use the explicit formulas listed after Lemma \ref{r=1formula}.  For $d+2\leq s \leq 2d+1$, we get
\begin{align*}
N_{d,r}(s-1)+N_{d,r}(s-d-1)&=(N_{d,1}(s-1)+(r-1)N_{d,1}(s-2d-1))+N_{d,1}(s-d-1)\\
 &=(N_{d,1}(s-1)+N_{d,1}(s-d-1))+(r-1)N_{d,1}(s-2d)\\
 &=N_{d,1}(s)+(r-1)N_{d,1}(s-2d)\\
 &=N_{d,r}(s).
\end{align*}
The second equality uses the uniformity of $N_{d,1}(s)$ on $s \leq 1$.  For $s \geq 2d+2$, we get
\begin{align*}
N_{d,r}(s-1)+N_{d,r}(s-d-1)&=(N_{d,1}(s-1)+(r-1)N_{d,1}(s-2d-1))\\
 &\hspace{0.4cm}+(N_{d,1}(s-d-1)+(r-1)N_{d,1}(s-3d-1))\\
 &=(N_{d,1}(s-1)+N_{d,1}(s-d-1))\\
 &\hspace{0.4cm}+(r-1)(N_{d,1}(s-2d-1)+N_{d,1}(s-3d-1))\\
 &=N_{d,1}(s)+(r-1)N_{d,1}(s-2d)\\
 &=N_{d,r}(s).
\end{align*}
This completes the casework and establishes the result.
\end{proof}

It is curious that this theorem seems not to have a natural combinatorial interpretation.  For the case of $r=1$, Sahin's proof of Theorem \ref{sahinthm} in \cite{sahin} establishes an explicit bijection by conditioning on whether or not $s-1$ is an element of $\beta \in X_{d,1,s}$.  For $r \geq 2$, however, the obvious arguments along these lines (conditioning on whether or not $\beta \in X_{d,r,s}$ contains, say, $1$, $s+r-1$, etc.) fail because the first condition of Lemma \ref{correspondence} is sensitive to changes in $s$.

\section{Generating functions and asymptotics}

\subsection{Generating functions}

For any positive integers $r\leq d$, we define the ordinary generating function
$$G_{d,r}(x)=\sum_{s=1}^{\infty}N_{d,r}(s)x^s.$$
(Note that the constant term is $0$, not $N_{d,r}(0)$.)  We first use Theorem \ref{sahinthmext} to find the generating functions $G_{d,1}(x)$.  Then, using the $r$-Reduction Theorem, we generalize this result to all $G_{d,r}(x)$ where $r \leq d$.

\begin{lemma}
For any positive integer $d$, the generating function $G_{d,1}(x)=\sum_{s=1}^{\infty}N_{d,1}(s)x^s$ is given by
$$G_{d,1}(x)=\frac{x(1-x^{d+1})}{(1-x-x^{d+1})(1-x)}.$$
\label{r=1genfunc}
\end{lemma}

\begin{proof}
Consider the auxiliary generating function
$$H_{d,1}(x)=\frac{x(1-x^d)}{1-x}+x^dG_{d,1}(x)=\sum_{s=1}^{d}x^s +\sum_{s=d+1}^{\infty} N_{d,1}(s-d)x^{s}=\sum_{s=1}^{\infty}N_{d,1}(s-d)x^s.$$
We can compute
\begin{align*}
x+xH_{d,1}(x)+x^{d+1}H_{d,1}(x)&=x+\sum_{s=2}^{\infty}N_{d,1}(s-d-1)x^s+\sum_{s=d+2}^{\infty}N_{d,1}(s-2d-1)x^s\\
 &=\sum_{s=1}^{d+1}x^s+\sum_{s=d+2}^{\infty} (N_{d,1}(s-d-1)+N_{d,1}(s-2d-1))x^s\\
 &=\sum_{s=1}^{d+1}x^s+\sum_{s=d+2}^{\infty} N_{d,1}(s-d)x^s\\
 &=H_{d,1}(x).
\end{align*}
The third equality follows from Theorem \ref{sahinthmext}.  We can now solve for $H_{d,1}(x)$ directly:
$$H_{d,1}(x)=\frac{x}{1-x-x^{d+1}}.$$
Finally, we can recover $G_{d,1}(x)$:
\begin{align*}
\frac{x(1-x^d)}{1-x}+x^dG_{d,1}(x)&=\frac{x}{1-x-x^{d+1}}\\
G_{d,1}(x)&=\frac{x(1-x^{d+1})}{(1-x-x^{d+1})(1-x)}.
\end{align*}
\end{proof}

Cancelling a factor of $1-x$ from the numerator and denominator yields the equivalent form
$$G_{d,1}(x)=\frac{x+x^2+\dots+x^{d+1}}{1-x-x^{d+1}}.$$
When $d=1$, we can recognize
$$G_{1,1}(x)=\frac{x+x^2}{1-x-x^2}$$
as the generating function for the shifted Fibonacci numbers, in accordance with the discussion in Section $3$.  Recall also from Section $3$ that $N_{d,1}(s)$ counts the $\{1, d+1\}$-restricted integer compositions of $s+d-1$.  Thus, $(\frac{1}{x})H_{d,1}(x)$ is the generating function for the number of $\{1, d+1\}$-restricted compositions of $s$.  See \cite{a-restricted} for an alternative derivation of $(\frac{1}{x})H_{d,1}(x)$ using the theory of compositions.
\\

We now extend Lemma \ref{r=1genfunc} to all $r \leq d$.

\begin{theorem}
For any positive integers $r \leq d$, the generating function $G_{d,r}(x)=\sum_{s=1}^{\infty}N_{d,r}(s)x^s$ is given by
$$G_{d,r}(x)=\frac{x(1+(r-1)x^d-rx^{d+1})}{(1-x-x^{d+1})(1-x)}.$$
\label{genfunc}
\end{theorem}
\pagebreak
\begin{proof}
We begin with
\begin{align*}
G_{d,1}(x)+(r-1)&\left(\frac{x^{d+1}(1-x^d)}{1-x}+x^{2d}G_{d,1}(x)\right)\\
 &\hspace{1cm}=\sum_{s=1}^{\infty}N_{d,1}(s)x^s+\sum_{s=d+1}^{\infty}(r-1)N_{d,1}(s-2d)x^s\\
 &\hspace{1cm}=\sum_{s=1}^{d} N_{d,1}(s)x^s+\sum_{s=d+1}^{\infty}(N_{d,1}(s)+(r-1)N_{d,1}(s-2d))x^s\\
 &\hspace{1cm}=\sum_{s=1}^{d} N_{d,r}(s)x^s+\sum_{s=d+1}^{\infty}N_{d,r}(s)x^s\\
 &\hspace{1cm}=G_{d,r}(s).
\end{align*}
The third equality comes from the $r$-Reduction Theorem.  Plugging in the formula from Lemma \ref{r=1genfunc} and simplifying gives
$$G_{d,r}(s)=\frac{x(1+(r-1)x^d-rx^{d+1})}{(1-x-x^{d+1})(1-x)}.$$
\end{proof}

Cancelling a factor of $1-x$ from the numerator and denominator yields the equivalent form
$$G_{d,r}(x)=\frac{x+x^2+\dots+x^{d-1}+rx^d}{1-x-x^{d+1}}.$$

\subsection{Asymptotics}

We can extract asymptotic formulas from these generating functions by analyzing their poles.  As usual, most of the work lies in the $r=1$ case.  Because our techniques are fairly standard, we defer these computations to Appendix A and state only the final results here.  In our notation, $f(n) \sim_n g(n)$ means that $\lim_{n \to \infty} \frac{f(n)}{g(n)}=1$.

\begin{lemma}
For each fixed positive integer $d$, we have the asymptotic
$$N_{d,1}(s) \sim_s \frac{w_d^3}{(1-w_d)(d+1-dw_d)} \left(\frac{1}{w_d}\right)^s$$
where $w_d\approx 1-\frac{\log(d+1)}{d+1}$ is the unique positive real root of $x^{d+1}+x-1$.
\label{r=1asymptotic}
\end{lemma}

We remark that $\frac{w_d^3}{(1-w_d)(d+1-dw_d)} \sim_d \frac{d+1}{(\log(d+1))^2}$ and $\frac{1}{w_d}=1+\frac{\log(d+1)}{d+1}+O\left(\frac{\log\log(d+1)}{d+1}\right)$.  The generalization to all $r \leq d$ is easy.

\begin{theorem}
For each fixed pair of positive integers $r \leq d$, we have the asymptotic
$$N_{d,r}(s) \sim_s \frac{w_d^3(1+(r-1)w_d^{2d})}{(1-w_d)(d+1-dw_d)} \left(\frac{1}{w_d}\right)^s.$$
\label{asymptotic}
\end{theorem}

\begin{proof}
The $r$-Reduction Theorem says that $N_{d,r}(s)=N_{d,1}(s)+(r-1)N_{d,1}(s-2d)$ for any $s\geq d+1$.  Plugging in the asymptotic formula from Lemma \ref{r=1asymptotic} and gathering like terms establishes the result.
\end{proof}

\section{The case of $r>d$}

In this section, we apply the techniques of the previous three sections to the case of $r>d$.  We begin by discussing why the case of general $r>d$ is fundamentally more complicated than the case of $r\leq d$.  For the remainder of the section, we focus on what appears to be the most approachable subcase of $r>d$: the case where $r$ is within $d$ of a multiple of $s$.  We sketch the proofs of exact formulas, recurrence relations, generating functions, and asymptotics for $r=ns-1$.  As we go, we discuss how these methods apply to all $r=ns\pm c$ where $1 \leq c\leq d$.

\subsection{The lay of the land}

The problem of computing $N_{d,r}(s)$ is fundamentally much more complicated when  $r>d$ than when $r \leq d$ for a variety of reasons.  First of all, Lemma \ref{finite} tells us that we now have to worry about $N_{d,r}(s)$ being infinite when $\gcd(s,r)>d$.  This fact makes a finding recurrence relation in the style of Theorem \ref{recurrence} (which relates $N_{d,r}(s)$ to $N_{d,r}(s-1)$ and smaller terms) unlikely, if not impossible.
\\

Second, Lemma \ref{betabounds} no longer bounds the size of the elements of sets $\beta\in X_{d,r,s}$.  When $\gcd(s,r)\leq d$, it can easily be shown with the Chicken McNugget Theorem that $\beta \subseteq \{1, 2, \dots, \frac{(s-\gcd(s,r))(s+r-\gcd(s,r))}{\gcd(s,r)}+2\gcd(s,r)\}$.  This upper bound, however, is not particularly useful: it grows very fast; the dependence on the greatest common divisor makes it volatile and tricky to work with; and the possibility of elements of $\beta$ being greater than $s+r$ means that the second condition of Lemma \ref{correspondence} is not trivially satisfied.
\\

Third, we do not know of any analogue of the $r$-Reduction Theorem, which so greatly simplified our work for $r\leq d$.  As such, we must address the $r>d$ case at a higher level of generality from the outset.
\\

The first two concerns are greatly reduced if we take $r$ to be close to a multiple of $s$, say, $r=ns\pm c$ for some $1 \leq c \leq d$.  In this case, the argument of Lemma \ref{betabounds} shows that any $\beta\in X_{d,ns+c,s}$ satisfies $\beta \subseteq \{1, 2, \dots, (n+1)s+c-1\}$ and any $\beta\in X_{d,ns-c,s}$ satisfies $\beta \subseteq \{1, 2, \dots, (n+1)s-1\}$.  
The fact that $\beta$ does not contain any elements larger than $s+r$ lets us not worry more about the $(s+r)$-core condition.  Furthermore, we draw inspiration from Straub's result \cite{straub} that $N_{1,ns-1}(s)=N_{1,n(s-1)-1}(s-1)+(n+1)N_{1,n(s-2)-1}(s-2)$ for all $s \geq 3$.

\subsection{Exact formulas}

We derive an exact formula for $N_{d,ns-1}(s)$ in the style of Lemma \ref{r=1formula}.  Although the casework is more complicated, the main idea remains the same.  We remark that the generalization to $N_{d,ns-c}(s)$ requires the addition of a few extra terms but is no harder.  The formula for $N_{d,ns+c}(s)$ is also very similar.  Readers familiar with abacus structure of core partitions will find the proof method especially natural.

\begin{theorem}
For any positive integers $n$ and $s$, we have
\begin{align*}
N_{1,ns-1}(s)=\sum_{\mu=0}^{\left\lceil \frac{s-1}{2}\right\rceil} \binom{s-\mu}{\mu} (n+1)^{\mu}.
\end{align*}
For any positive integers $d>1$, $n$, and $s$, we have
\begin{align*}
N_{d,ns-1}(s)&=\sum_{\mu=0}^{\left\lceil \frac{s-d}{d+1}\right\rceil} \binom{s-d\mu}{\mu} (n+1)^{\mu}+n\sum_{\mu=0}^{\left\lceil \frac{s-2d-1}{d+1}\right\rceil} \binom{s-d-d\mu-1}{\mu} (n+1)^{\mu}\\
 &+n\sum_{k=1}^{d-1}\sum_{\mu=0}^{\left\lceil \frac{s-2d-k-2}{d+1}\right\rceil} \binom{s-d-k-d\mu-2}{\mu} (n+1)^{\mu}\\
 &+\sum_{\ell=1}^{d-2}(n+1)\sum_{\mu=0}^{\left\lceil \frac{s-2d-1}{d+1}\right\rceil} \binom{s-d-d\mu-1}{\mu} (n+1)^{\mu}\\
 &+\sum_{\ell=1}^{d-2}(n+1)\sum_{k=1}^{\ell}\sum_{\mu=0}^{\left\lceil \frac{s-3d+\ell-k-1}{d+1}\right\rceil} \binom{s-2d+\ell-k-d\mu-1}{\mu} (n+1)^{\mu}.
\end{align*}
\label{r>dformula}
\end{theorem}

\begin{proof}
As usual, we count the sets $\beta \in X_{d,ns-1,s}$.   From above, we have $\beta \subseteq \{1, 2, \dots, (n+1)s-2\}\setminus \{s, 2s, \dots, ns\}$.  Recall that $x \in \beta$ for $x>s$ requires $x-s\in \beta$.  Hence, the elements of $\beta \cap \{s+1, s+2, \dots, (n+1)s-2\}$ are restricted to the equivalence classes modulo $s$ of $\beta \cap \{1, 2, \dots, s-1\}$.  Suppose we have a $d$-th order twin-free set $\gamma\subseteq \{1, 2, \dots, s+d-1\}\setminus \{s\}$, and let $\delta=\gamma \cap \{d, d+1, \dots, s+d-1\}$.  For any $\{x_1, \dots, x_m\} \subseteq \delta$ and nonnegative integers $a_1, \dots, a_m$, it is clear that
$$\gamma \cup \bigcup_{i=1}^m \{x_i+s, x_i+2s, \dots x_i+a_is\}$$
is a valid $\beta$-set as long as each $x_i+a_i s \leq (n+1)s-2$.  (We needed to check elements up to $s+d-1$ in order to ensure that the $d$-th order twin-free condition is preserved for the ``wrap-around'' effect of elements near $s$.)  In fact, any $\beta \in X_{d,ns-1,s}$ can be written in this form.  We distinguish five cases for possible sets $\eta=\gamma \cap \{1, 2, \dots, s-1\}$, each of which contributes to $N_{d,ns-1}(s)$.
\\

If $\eta \cap \{s-d+1, s-d+2, \dots, s-1\}=\emptyset$, then we can freely include any $d$-th order twin-free subset of $\{1, 2, \dots, s-d\}$ in $\eta$.  Because of the ``buffer'' space in $\{s-d+1, s-d+2, \dots, s-1\}$, we can bypass the consideration of $\delta$, and we see that any elements of $\eta$ can ``propagate'' to larger elements of their equivalence classes modulo $s$.  In particular, there are $n+1$ options for how far each $x \in \eta$ propagates.  Conditioning on the number of elements in $\eta$ (\'a la Lemma \ref{r=1formula}) gives
$$\sum_{\mu=0}^{\left\lceil \frac{s-d}{d+1}\right\rceil} \binom{(s-d)+d-d\mu}{\mu} (n+1)^{\mu}=\sum_{\mu=0}^{\left\lceil \frac{s-d}{d+1}\right\rceil} \binom{s-d\mu}{\mu} (n+1)^{\mu}$$
ways to do this.  We remark that when $d=1$, this condition on $\eta$ is always satisfied, so the remaining cases are relevant only for $d>1$.
\\

If $s-1 \in \eta$ and $\eta \cap \{1, 2, \dots, d-1\}=\emptyset$, then we know that the other elements of $\gamma$ are all in $\{d, d+1, \dots, s-d-2\}$.  Once again, we can choose any $d$-th order twin-free subset of $\{d, d+1, \dots, s-d-2\}$, and these elements can propagate freely upwards.  We note that there are only $n$ options for how high the element $s-1$ propagates, for a total of
$$n\sum_{\mu=0}^{\left\lceil \frac{s-2d-1}{d+1}\right\rceil} \binom{s-d-d\mu-1}{\mu} (n+1)^{\mu}.$$
\\

If $s-1 \in \eta$ and $k \in \eta$ for some $1 \leq k \leq d-1$, then the other elements of $\gamma$ are all in $\{k+d+1, k+d+2, \dots, s-d-2\}$.  (Because of the $d$-th order twin-free condition, $\eta$ contains at most $1$ element smaller than $d$.)  Note that $k$ cannot propagate upwards at all due to the presence of the element $s-1$.  As above, our total is
$$n\sum_{k=1}^{d-1}\sum_{\mu=0}^{\left\lceil \frac{s-2d-k-2}{d+1}\right\rceil} \binom{s-d-k-d\mu-2}{\mu} (n+1)^{\mu}.$$
\\

If $s-d+\ell \in \eta$ for some $1 \leq \ell \leq d-2$ and $\eta \cap \{1, 2, \dots, \ell\}=\emptyset$, then the other elements of $\eta$ are all in $\{\ell+1, \dots, s-2d+\ell-1\}$.  (Because of the $d$-th order twin-free condition, $\eta$ contains at most $1$ element in $\{s-d+1, \dots, s-1\}$.)  There are $n+1$ possibilities for how high $s-d+\ell$ propagates, so our total is
$$\sum_{\ell=1}^{d-2}(n+1)\sum_{\mu=0}^{\left\lceil \frac{s-2d-1}{d+1}\right\rceil} \binom{s-d-d\mu-1}{\mu} (n+1)^{\mu}.$$
We remark that this term is $0$ when $d\leq 2$.  For $d \geq 3$, the expression simplifies to
$$(d-2)(n+1)\sum_{\mu=0}^{\left\lceil \frac{s-2d-1}{d+1}\right\rceil} \binom{s-d-d\mu-1}{\mu} (n+1)^{\mu}.$$
\\

If $s-d+\ell \in \eta$ for some $1 \leq \ell \leq d-2$ and $k \in \eta$ for some $1\leq k \leq \ell$, then the other elements of $\eta$ are all in $\{k+d+1, \dots, s-2d+\ell-1\}$.  Note that $k$ cannot propagate upwards due to the element $s-d+\ell$.  The number of possibilities is
$$\sum_{\ell=1}^{d-2}(n+1)\sum_{k=1}^{\ell}\sum_{\mu=0}^{\left\lceil \frac{s-3d+\ell-k-1}{d+1}\right\rceil} \binom{s-2d+\ell-k-d\mu-1}{\mu} (n+1)^{\mu}.$$
\\

Summing the contributions from these five cases gives the desired formula.  When $s\leq 2d$, a few subcases of the third through fifth cases are prohibited by the \linebreak $d$-th order twin-free condition.  It is not difficult to verify that the offending term vanishes whenever this happens.
\end{proof}

We remark that when $d\leq 2$, the fourth and fifth contributions vanish.  When $d=1$, the third contribution also vanishes, which yields a much simpler formula.

\subsection{Recurrence relations}

Although the formula in Theorem \ref{r>dformula} is long, the fact that all of the terms look very similar is a saving grace.  The following lemma examines a single generic term.

\begin{lemma}
For any positive integers $d$, $n$, and $t$, let
$$\mathcal{B}_{d,n}(t)=\sum_{\mu=0}^{\left\lceil \frac{t}{d+1}\right\rceil}\binom{t+d-d\mu}{\mu}(n+1)^{\mu}.$$
Then $\mathcal{B}_{d,n}(t-1)+(n+1)\mathcal{B}_{d,n}(t-d-1)=\mathcal{B}_{d,n}(t).$
\label{binomial}
\end{lemma}

\begin{proof}
The identity follows from algebraic manipulations:
\begin{align*}
 &\mathcal{B}_{d,n}(t-1)+(n+1)\mathcal{B}_{d,n}(t-d-1)\\
 &\hspace{1cm}=\sum_{\mu=0}^{\left\lceil \frac{t-1}{d+1}\right\rceil}\binom{(t-1)+d-d\mu}{\mu}(n+1)^{\mu}+(n+1)\sum_{\mu=0}^{\left\lceil \frac{t-d-1}{d+1}\right\rceil}\binom{(t-d-1)+d-d\mu}{\mu}(n+1)^{\mu}\\
 &\hspace{1cm}=\sum_{\mu=0}^{\left\lceil \frac{t-1}{d+1}\right\rceil}\binom{t-1+d-d\mu}{\mu}(n+1)^{\mu}+\sum_{\mu=1}^{\left\lceil \frac{t}{d+1}\right\rceil}\binom{t-1-d(\mu-1)}{\mu-1}(n+1)^{\mu}\\
 &\hspace{1cm}=\sum_{\mu=0}^{\left\lceil \frac{t}{d+1}\right\rceil}\left(\binom{t-1+d-d\mu}{\mu}+\binom{t-1+d-d\mu}{\mu-1}\right)(n+1)^{\mu}\\
 &\hspace{1cm}=\sum_{\mu=0}^{\left\lceil \frac{t}{d+1}\right\rceil}\binom{t+d-d\mu}{\mu}(n+1)^{\mu}\\
 &\hspace{1cm}=\mathcal{B}_{d,n}(t).\\
\end{align*}
In the third line, extending the ranges of the sums adds only terms equal to $0$.
\end{proof}

We can use this lemma to derive a generalization of Straub's recurrence relation.

\begin{theorem}
For any positive integers $d$ and $n$ and any integer $s \geq 3d+2$, we have the recurrence relation
$$N_{d,ns-1}(s)=N_{d,n(s-1)-1}(s-1)+(n+1)N_{d,n(s-d-1)-1}(s-d-1).$$
\label{r>drecurrence}
\end{theorem}

\begin{proof}
Apply Lemma \ref{binomial} to each term of the formula in Theorem \ref{r>dformula} separately.  The condition $s\geq 3d+2$ guarantees that the upper bounds of the sums fall within the scope of Lemma \ref{binomial}.  Note that in the third through fifth terms of Theorem \ref{r>dformula}, the outer sums do not depend on $s$.
\end{proof}

We remark that for general $N_{d, ns \pm c}(s)$, the the recurrence $$N_{d,ns\pm c}(s)=N_{d,n(s-1)\pm c}(s-1)+(n+1)N_{d,n(s-d-1) \pm c}(s-d-1)$$ holds for all sufficiently large $s$.  This is true because the formulas for $N_{d, ns \pm c}(s)$ are composed of the same ``types'' of terms as the formula for $N_{d,ns-1}(s)$.

\subsection{Generating functions and asymptotics}

For any positive integers $c$, $d$, and $n$ with $c\leq d$, we define the generating functions
$$G_{d,n,\pm c}(x)=\sum_{s=1}^{\infty}N_{d,ns\pm c}(s) x^s.$$
We know from the theory of linear recurrences that $G_{d,n,\pm c}(x)$ is some rational function with denominator $f_{d,n}(x)=1-x-(n+1)x^{d+1}$.  Although we do not compute the numerator in this paper, we see no fundamental obstruction to finding it by computing the values of $N_{d,ns\pm c}(s)$ for small $s$.
\\

Arguments in the style of Appendix A show that
$$N_{d,ns\pm c}(s) \sim_s k_{n,d,\pm c} w_{d,n}^{-s}$$
for some positive constant $k_{n,d,\pm c}$ depending on $c$ and the sign of the deviation as well as on $d$ and $n$.

\section{Open problems}

We conclude this paper with a few open problems.

\begin{enumerate}
\item Find a simple bijective proof of Theorem \ref{recurrence}.
\item What is the ``correct'' interpretation of $N_{d,r}(s)$ for $s \leq 0$ when $r>d$?
\item Are other special cases of $r>d$ tractable?  Small values of $r$ (relative to $d$) seem like a natural place to start.  Values of $r$ where $2r$ is close to a multiple of $s$ could also be interesting.  Finally, the case $c=d+1$, where $s$ is not divisible by $d+1$, is potentially promising (see, e.g., \cite{baek, yan} for $d=1$) but likely quite difficult in general.
\end{enumerate}

\appendix
\section{Computation of asymptotics }

This appendix is devoted to proving Lemma \ref{r=1asymptotic}.  We begin by establishing some properties of the polynomial $f_d(z)=z^{d+1}+z-1$.  The following lemma shows that the nonremovable singularities of $G_{d,1}(z)$ are simple poles at the roots of $f_d(z)$.  (The singularity at $z=1$ is clearly removable.)

\begin{lemma}
For each positive integer $d$, the roots of $f_d(z)=z^{d+1}+z-1$ are all distinct.  Moreover, $f_d(z)$ and $z(1-z^{d+1})$ have no common roots.
\label{distinctroots}
\end{lemma}

\begin{proof}
For the first statement, we show that $f_d(z)$ and $f'_d(z)$ are relatively prime as elements of $\mathbb{Q}[z]$.  We begin with
\begin{align*}
\gcd(f_d(z), f'_d(z))&=\gcd(z^{d+1}+z-1, (d+1)z^d+1)\\
 &=\gcd(\frac{d}{d+1}z-1, (d+1)z^d+1).
\end{align*}
The only root of $\frac{d}{d+1}z-1$ is $\frac{d+1}{d}$, but $(d+1)(\frac{d+1}{d})^d+1>1$ implies that $\frac{d+1}{d}$ is not a root of $(d+1)z^d+1$.  Hence, $\gcd(f_d(z), f'_d(z))=\gcd(\frac{d}{d+1}z-1, (d+1)z^d+1)=1$, and $f_d(z)$ has no repeated roots.
\\

For the second statement, suppose $z_0$ is a root of $z(1-z^{d+1})$.  Then either $z_0=0$, which is clearly not a root of $f_d(z)$, or $z_0$ satisfies $1-z_0^{d+1}=0$.  In the latter case, $z_0\neq 0$, and $f_d(z_0)=z_0^{d+1}+z_0-1=z_0\neq 0$ shows that $z_0$ is not a root of $f_d(z)$.
\end{proof}

We now describe the locations of the roots of $f_d(z)$.

\begin{lemma}
For each positive integer $d$, the polynomial $f_d(z)$ has a unique positive real root $w_d$, and all other roots of $f_d(z)$ have modulus strictly larger than $w_d$.
\label{posroot}
\end{lemma}

\begin{proof}
We first compute $f_d(0)=0+0-1=-1<0$ and $f_d(1)=1+1-1=1>0$.  By the Intermediate Value Theorem, $f_d(z)$ has a real root in $(0,1)$.  The fact that $ f'_d(z)=(d+1)z^d+1>1$ for all positive real $z$ precludes the existence of a second positive real root and hence shows uniqueness.
\\

Now, suppose $z_0 \neq w_d$ is any other (possibly complex) root of $f_d(z)$.  Write $z_0=Re^{i\theta}$ where $R>0$ (we know that $0$ is not a root) and $0<\theta<2\pi$.  We have $z_0^{d+1}+z_0-1=0$ and $Re^{(d+1)i\theta}+Re^{i\theta}=1$.  By the triangle inequality, $R^{d+1}+R\geq 1$, with equality only when $Re^{i\theta}$ and $Re^{(d+1)i\theta}$ have the same argument.  But in this equality case, $Re^{(d+1)i\theta}+Re^{i\theta}$ has argument $\theta$, which contradicts this expression equaling $1$, so we can conclude that the inequality is strict.  Then $f_d(R)=R^{d+1}+R-1>0$, and the monotonicity of $f_d(z)$ on $(0, \infty)$ implies that $R>w_d$.  Hence, $w_d$ is the root with strictly smallest modulus.
\end{proof}

The next task is approximating $w_d$.  We preface the following lemma with a heuristic explanation.  When $d$ becomes large, $w_d^{d+1}$ is very small, and we expect $w_d$ to approach $1$.  Thus, we write $w_d=1-\frac{\varepsilon(d)}{d+1}$ and note that $w_d^{d+1}=(1-\frac{\varepsilon(d)}{d+1})^{d+1}\approx e^{-\varepsilon(d)}$ for large $d$.  Plugging this approximation into $f_d(w_d)=0$ gives $e^{-\varepsilon(d)}+(1-\frac{\varepsilon(d)}{d+1})-1\approx 0$ and hence $d+1 \approx \varepsilon(d) e^{\varepsilon(d)}$.  Thus, we expect $\varepsilon(d)$ to be close to the real branch evaluation of $W(d+1)$, where the Lambert W-function $W(z)$ is the inverse of the map $z \mapsto ze^z$.  It is well known (see, e.g., \cite{lambert}) that $W(z)\approx \log(z)-\log\log(z)$ for real $z\geq 3$.  Hence, we approximate $w_d$ by $w_d^{\ast}=1-\frac{\log(d+1)}{d+1}+\frac{\log\log(d+1)}{d+1}$.  We now show that this is in fact a good approximation.

\begin{lemma}
For any positive integer $d$, we have the bounds
$$w_d=1-\frac{\log(d+1)}{d+1}+O\left(\frac{\log\log(d+1)}{d+1}\right).$$
\label{wdbounds}
\end{lemma}

\begin{proof}
First, we compute
\begin{align*}
\log\left((w_d^{\ast})^{d+1}\right) &=(d+1)\log\left(1-\frac{\log(d+1)}{d+1}+\frac{\log\log(d+1)}{d+1}\right)\\
 &=(d+1)\left(-\frac{\log(d+1)}{d+1}+\frac{\log\log(d+1)}{d+1}+O\left(\frac{(\log(d+1)-\log\log(d+1))^2}{(d+1)^2}\right)\right)\\
 &=-\log(d+1)+\log\log(d+1)+o((d+1)^{c-1})
\end{align*}
for any small $c>0$.  The second equality comes from the Taylor expansion of $\log(1+x)$.  Exponentiating gives
$$(w_d^{\ast})^{d+1}=\frac{\log(d+1)}{d+1}(1+o((d+1)^{c-1}))=\frac{\log(d+1)}{d+1}+o((d+1)^{2c-2}).$$
We now simply plug $w_d^{\ast}$ into $f_d(x)$:
\begin{align*}
f_d(w_d^{\ast})&=\frac{\log(d+1)}{d+1}+o((d+1)^{2c-2})+\left(1-\frac{\log(d+1)}{d+1}+\frac{\log\log(d+1)}{d+1}\right)-1\\
 &=O\left(\frac{\log\log(d+1)}{d+1}\right)
\end{align*}
Recall from the proof of Lemma \ref{posroot} that $f_d'(x)>1$ for all positive $x$.  Hence,
$$|w_d-w_d^{\ast}|<|f_d(w_d)-f_d(w_d^{\ast})|=|f_d(w_d^{\ast})|=O\left(\frac{\log\log(d+1)}{d+1}\right).$$
Finally, we conclude that
$$w_d=1-\frac{\log(d+1)}{d+1}+O\left(\frac{\log\log(d+1)}{d+1}\right).$$
\end{proof}

We now transition into the complex analysis portion of this section.

\begin{proposition}
For any positive integer $d$ and any root $w$ of $f_d(z)$, the residue of $G_{d,1}(z)$ at $z=w$ is given by
$$\Res_{w}\left(G_{d,1}\right)=-\frac{w^3}{(1-w)(d+1-dw)}.$$
\label{residue}
\end{proposition}

\begin{proof}
The computation is straightforward:
\begin{align*}
\Res_{w}\left(G_{d,1}\right)&=\frac{w(1-w^{d+1})}{1-w} \lim_{z \to w} \frac{z-w}{1-z-z^{d+1}}\\
 &=\frac{w(1-w^{d+1})}{1-w} \lim_{z \to w} \frac{1}{-1-(d+1)z^d}\\
 &=-\frac{w(1-w^{d+1})}{(1-w)(1+(d+1)w^d)}\\
 &=-\frac{w^3}{(1-w)(d+1-dw)}.
\end{align*}
The second equality follows from L'Hospital's Rule, and the third equality holds because $w$ is not a root of $(d+1)z^d+1$ (from Lemma \ref{distinctroots}).  The fourth equality uses $w^{d+1}+w-1=0$ to substitute for $w^{d+1}$ and $w^d$.
\end{proof}

We can finally prove Lemma \ref{r=1asymptotic}.

\begin{proof}[Proof of Lemma \ref{r=1asymptotic}]
As is well-known, the asymptotics of the coefficients of $G_{d,1}(z)$ are controlled by the pole with the smallest modulus.  So, for each fixed $d$,
$$N_{d,1}(s) \sim_s -Res_{w_d}(G_{d,1})\left(\frac{1}{w_d}\right)^s=\frac{w_d^3}{(1-w_d)(d+1-dw_d)} \left(\frac{1}{w_d}\right)^s.$$
\end{proof}

\vspace{0.5cm}

\textbf{Acknowledgments.} This research was conducted at the University of Minnesota, Duluth REU and was supported by NSF/DMS grant 1650947 and NSA grant H98230-18-1-0010.  The author wishes to thank Joe Gallian for suggesting this problem and Sam Judge, Aaron Berger, and Joe Gallian for reading paper drafts.  The author also benefited from discussions with Mitchell Lee and Ashwin Sah.  Finally, the author is grateful to Alan Peng for pointing out a subtlety in the $d=1$ case of Theorem~\ref{r>dformula}.

\end{document}